\documentclass[11pt,letterpaper]{amsart}

\usepackage{amssymb}
\usepackage{enumerate}

\newtheorem{theorem}{Theorem}[section]
\newtheorem{proposition}[theorem]{Proposition}
\newtheorem{corollary}[theorem]{Corollary}
\newtheorem{lemma}[theorem]{Lemma}
\newtheorem{claim}[theorem]{Claim}
\newtheorem{conjecture}[theorem]{Conjecture}

\newtheorem*{weil}{Weil's Theorem}
\newtheorem*{bezout}{Bezout's Theorem}

\theoremstyle{definition}
\newtheorem{definition}[theorem]{Definition}
\newtheorem*{remark}{Remark}

\DeclareMathOperator{\re}{Re}

\DeclareMathOperator{\Tr}{Tr}

\DeclareMathOperator{\wt}{wt}
\newcommand{\F}{\mathbb{F}}
\newcommand{\K}{\mathbb{K}}

\newcommand{\Cn}{\mathbb{C}}
\renewcommand{\P}{\mathbb{P}}

\newcommand{\Z}{\mathbb{Z}}
\newcommand{\A}{\mathcal{A}}
\newcommand{\B}{\mathcal{B}}
\newcommand{\C}{\mathcal{C}}
\newcommand{\D}{\mathcal{D}}
\newcommand{\Pc}{\mathcal{P}}

\newcommand{\e}{\epsilon}

\newcommand{\abs}[1]{\lvert#1\rvert}

\newcommand{\Biggabs}[1]{\Bigg\lvert#1\Bigg\rvert}

\begin{document}

\title{Planar functions over fields of characteristic two}

\date{29 January 2013 (revised 19 December 2013)}

\author{Kai-Uwe Schmidt}
\address{Faculty of Mathematics, Otto-von-Guericke University, Universit\"atsplatz~2, 39106 Magdeburg, Germany}
\email{kaiuwe.schmidt@ovgu.de}

\author{Yue Zhou}
\address{Faculty of Mathematics, Otto-von-Guericke University, Universit\"atsplatz~2, 39106 Magdeburg, Germany}
\curraddr{Department of Mathematics and System Sciences, College of Science, National University of Defense Technology, Changsha, China}
\email{yue.zhou.ovgu@gmail.com}

\begin{abstract}
Classical planar functions are functions from a finite field to itself and give rise to finite projective planes. They exist however only for fields of odd characteristic. We study their natural counterparts in characteristic two, which we also call planar functions. They again give rise to finite projective planes, as recently shown by the second author. We give a characterisation of planar functions in characteristic two in terms of codes over $\Z_4$. We then specialise to planar monomial functions $f(x)=cx^t$ and present constructions and partial results towards their classification. In particular, we show that $t=1$ is the only odd exponent for which $f(x)=cx^t$ is planar (for some nonzero $c$) over infinitely many fields. The proof techniques involve methods from algebraic geometry.
\end{abstract}

\maketitle


\section{Introduction}

A function $f:\F_q\to\F_q$ is \emph{planar} if
\begin{equation}
x\mapsto f(x+\e)-f(x)   \label{eqn:def_pn}
\end{equation}
is a permutation of $\F_q$ for each $\e\in\F_q^*$. Planar functions have been introduced by Dembowski and Ostrom~\cite{DemOst1968} to construct finite projective planes and arise in many other contexts. For example, Ganley and Spence~\cite{GanSpe1975} showed that planar functions give rise to certain relative difference sets, Nyberg and Knudsen~\cite{NybKnu1993}, among others, studied planar functions (under the synonym \emph{perfect nonlinear functions}) for applications in cryptography, and Carlet, Ding, and Yuan~\cite{CarDingYua2005}, among others, used planar functions to construct error-correcting codes.
\par
Planar functions cannot exist in characteristic two since, if $q$ is even and $x$ is a solution to $f(x+\e)-f(x)=a$ for $a\in\F_q$, then so is $x+\e$. This is the motivation to define a function $f:\F_q\to\F_q$ to be \emph{almost perfect nonlinear} if~\eqref{eqn:def_pn} is a $2$-to-$1$ map. Such functions have also been studied extensively for applications in cryptography and coding theory (see Carlet, Charpin, and Zinoviev~\cite{CarChaZin1998}, for example). However, there is no apparent link between almost perfect nonlinear functions and finite projective planes.
\par
Recently, the second author proposed~\cite{Zho2013} a concept to overcome the problem that there is no planar function in characteristic two. The definition of a planar function has to be modified as follows. 
\begin{definition}
A function $f:\F_{2^n}\to\F_{2^n}$ is \emph{planar} if
\begin{equation}
x\mapsto f(x+\e)+f(x)+\e x   \label{eqn:planar_def}
\end{equation}
is a permutation of $\F_{2^n}$ for each $\e\in\F_{2^n}^*$.
\end{definition}
Such functions share many of the properties of planar functions in odd characteristic. The next section, which is independent of the rest of this paper, provides further background on planar functions in characteristic two and discusses connections to finite geometries and coding theory.
\par
Every function from $\F_{2^n}$ to itself can be uniquely written as a polynomial function of degree strictly less than $2^n$. We consider the simplest nontrivial polynomial functions, namely monomial functions $x\mapsto cx^t$ for some $c\in\F_{2^n}^*$ and some integer $t$. Such functions are often preferred in applications. We are interested in those exponents $t$ that give rise planar functions.
\begin{definition}
An integer $t$ satisfying $0<t<2^n$ is a \emph{planar exponent} of $\F_{2^n}$ if the function $x\mapsto cx^t$ is planar on $\F_{2^n}$ for some $c\in\F_{2^n}^*$.
\end{definition}
\par
Trivially, $2^k$ is a planar exponent of all fields $\F_{2^n}$ satisfying $n>k$. A nontrivial example is given in Theorem~\ref{thm:planar_monomial_1}, which shows that $2^k+1$ is a planar exponent of $\F_{4^k}$. In an earlier version of this paper, we conjectured that $4^k(4^k+1)$ is a planar exponent of $\F_{64^k}$. This was subsequently proved by Scherr and Zieve~\cite{SchZie2013}. We conjecture that these examples, summarised in Table~\ref{tab:planar_exponents}, form the complete list of planar exponents.
\begin{table}[h]
\centering
\caption{Conjectured complete list of planar exponents of $\F_{2^n}$}
\label{tab:planar_exponents}
\begin{tabular}{c|c|c}
\hline
Exponent $t$ & Condition & Reference\\\hline\hline
$2^k$        & none    & trivial\\
$2^k+1$      & $n=2k$  & Theorem~\ref{thm:planar_monomial_1}\\
$4^k(4^k+1)$ & $n=6k$  & \cite[Theorem 1.1]{SchZie2013} \\\hline
\end{tabular}
\end{table}
\par
As in odd characteristic, the classification of planar monomials in characteristic two seems to be a challenging problem. This motivates us to study the relaxed problem of classifying those numbers that are planar exponents of $\F_{2^n}$ for infinitely many $n$. The only known such numbers are the powers of $2$ and we conjecture that there are no more. Our main result is the following.
\par
\begin{theorem}
\label{thm:exceptional_monomials_odd}
If $t$ is an odd planar exponent of $\F_{2^n}$ for infinitely many $n$, then $t=1$.
\end{theorem}
\par
The problem of classifying the numbers that are planar exponents of $\F_{2^n}$ for infinitely many $n$ parallels the problem of classifying monomial functions $x\mapsto x^t$ on $\F_{2^n}$ that are almost perfect nonlinear for infinitely many~$n$. To attack this problem, Janwa, McGuire, and Wilson~\cite{JanMcGWil1995} proposed to use ideas from algebraic geometry. These ideas were further developed by Jedlicka~\cite{Jed2008} and Hernando and McGuire~\cite{HerMcG2011}, leading to a complete solution. The same approach has been used by Hernando and McGuire~\cite{HerMcG2012} to prove a conjecture on monomial hyperovals in projective planes and by Leducq~\cite{Led2012} and Hernando, McGuire, and Monserrat~\cite{HerMcGMon2013} to give partial results towards a classification of monomial functions $x\mapsto x^t$ on $\F_{p^n}$ (with $p$ odd) that are planar for infinitely many~$n$ (which was recently completed by Zieve~\cite{Zie2013} using different techniques). We use a similar approach to prove Theorem~\ref{thm:exceptional_monomials_odd}, though our proof requires several extra ideas.


\section{Background and Motivation}
\label{sec:background}

\subsection{Relative difference sets and finite geometries}

Let $G$ be a finite group and let $N$ be a subgroup of $G$. A subset $D$ of $G$ is a \emph{relative difference set} with parameters $(\abs{G}/\abs{N},\abs{N},\abs{D},\lambda)$ and \emph{forbidden subgroup} $N$ if the list of nonzero differences of $D$ comprises every element in $G\setminus N$ exactly $\lambda$ times. We are interested in relative difference sets $D$ with parameters $(q,q,q,1)$ and a normal forbidden subgroup, in which case a classical result due to Ganley and Spence~\cite[Theorem~3.1]{GanSpe1975} shows that $D$ can be uniquely extended to a finite projective plane.
\par
It is known~\cite{Gan1976},~\cite{Jun1987} that, for even $q$, a relative difference set with parameters $(q,q,q,1)$ in an abelian group necessarily satisfies $q=2^n$ for some integer~$n$ and is a subset of $\Z_4^n$ (where $\Z_4=\Z/4\Z$) and the forbidden subgroup is $2\Z_4^n$. This fact was the motivation for the second author to study~\cite{Zho2013} such relative difference sets, which then led to the notion of planar functions over fields of characteristic two.
\par
We shall follow an approach that is slightly different from that in~\cite{Zho2013} and identify $\Z_4^n$ with the additive group of the Galois ring $R_n$ of characteristic $4$ and cardinality $4^n$. We recall some basic facts about such Galois rings (see~\cite{Nec1991} or~\cite{HamKumCalSloSol1994}, for example). The unit group $R_n\setminus 2R_n$ of $R_n$ contains a cyclic subgroup $\Gamma(R_n)^*$ of size $2^n-1$ and $\Gamma(R_n)=\Gamma(R_n)^*\cup\{0\}$ is called the \emph{Teichmuller set} in $R_n$. We define addition on $\Gamma(R_n)$ by
\begin{equation}
x\oplus y=x+y+2\sqrt{xy}   \label{eqn:def_oplus}
\end{equation}
(where $+$ is addition in $R_n$). Then $(\Gamma(R_n),\oplus,\,\cdot\,)$ is a finite field with $2^n$ elements~\cite[Statement~2]{Nec1991}. Every $y\in R_n$ can be written uniquely in the form $y=a+2b$ for $a,b\in\Gamma(R_n)$.
\par
It is now an easy exercise to show that a relative difference set in $R_n$ with parameters $(2^n,2^n,2^n,1)$ can always be written as
\begin{equation}
D=\{x+2\sqrt{f(x)}:x\in\Gamma(R_n)\},   \label{eqn:diff_set}
\end{equation}
where $f$ is some function from $\Gamma(R_n)$ to itself. The following result characterises the functions $f$ for which~\eqref{eqn:diff_set} is a relative difference set.
\begin{theorem}
\label{thm:diff_set}
The set $D$, given in~\eqref{eqn:diff_set}, is a relative difference set with parameters $(2^n,2^n,2^n,1)$ and forbidden subgroup $2R_n$ if and only if $f$ is planar.
\end{theorem}
\begin{proof}
By definition, $D$ is a relative difference set with parameters $(2^n,2^n,2^n,1)$ and forbidden subgroup $2R_n$ if and only if, for every $c\in R\setminus 2R$, the equation
\[
(x+2\sqrt{f(x)})-(y+2\sqrt{f(y)})=c
\]
has exactly one solution $(x,y)\in\Gamma(R_n)\times\Gamma(R_n)$. Equivalently, writing $c=a+2b$ for $a\in\Gamma(R_n)^*$ and $b\in\Gamma(R_n)$, the two equations
\begin{gather*}
x\oplus y=a\\
\sqrt{f(x)}\oplus\sqrt{f(y)}\oplus \sqrt{xy}\oplus y=b
\end{gather*}
hold simultaneously for exactly one pair $(x,y)\in \Gamma(R_n)\times \Gamma(R_n)$. This in turn holds if and only if the mapping
\[
x\mapsto f(x\oplus a)\oplus f(x)\oplus ax
\]
is a permutation of $\Gamma(R_n)$ for every $a\ne 0$.
\end{proof}
\par
\begin{remark}
Theorem~\ref{thm:diff_set} is essentially equivalent to~\cite[Theorem~2.1]{Zho2013}, which avoids using Galois rings at the cost of a more delicate proof.
\end{remark}
\par
Let $\chi:R_n\to\Cn$ be a character of the additive group of $R_n$. For later reference, we recall the following standard result (see~\cite[Ch.~1]{Pot1995}, for example): $D$ is a relative difference set in $R_n$ with forbidden subgroup $2R_n$ if and only~if
\begin{equation}
\Biggabs{\sum_{x\in D}\chi(x)}^2=\begin{cases}
4^n & \text{for $\chi$ principal}\\
0   & \text{for $\chi$ not principal, but principal on $2R_n$}\\
2^n & \text{otherwise}.
\end{cases}   \label{eqn:diff_set_char}
\end{equation}

\subsection{Coding theory}

We assume that the reader is familiar with the basic terminology of coding theory, in particular of the theory of codes over $\Z_4$. Otherwise, we advise to consult the seminal paper~\cite{HamKumCalSloSol1994}.
\par
Let $f$ be a function from $\F_{2^n}$ to itself satisfying $f(0)=0$ and let $\alpha$ be a generator of $\F_{2^n}^*$. It is well known (see~\cite[Theorem~5]{CarChaZin1998}, for example) that for $n\ge 4$ the code over $\F_2$ having parity check matrix
\begin{equation}
\begin{bmatrix}
1    & \alpha    & \alpha^2    & \cdots & \alpha^{2^n-2}\\
f(1) & f(\alpha) & f(\alpha^2) & \cdots & f(\alpha^{2^n-2})
\end{bmatrix}   \label{eqn:pc_matrix_BCH}
\end{equation}
has minimum (Hamming) distance $3$, $4$, or $5$, where the value~$5$ occurs if and only if $f$ is almost perfect nonlinear. We shall provide a similar characterisation for planar functions in characteristic two. 
\par
Let $f$ be a function from $\Gamma(R_n)$ to itself and let $\beta$ be a generator of $\Gamma(R_n)^*$. Consider the code $\C_f$ over $\Z_4$ having parity check matrix
\[
\begin{bmatrix}
1            & 1              & 1                      & \cdots & 1\\
2\sqrt{f(0)} & 1+2\sqrt{f(1)} & \beta+2\sqrt{f(\beta)} & \cdots & \beta^{2^n-2}+2\sqrt{f(\beta^{2^n-2})}
\end{bmatrix}.
\]
This code and its dual are free $\Z_4$-modules of rank $4^{2^n-n-1}$ and $4^{n+1}$, respectively.
\par
We remind the reader that the \emph{Lee weights} of $0,1,2,3\in\Z_4$ are $0,1,2,1$, respectively, and the \emph{Lee weight} $wt_L(c)$ of $c\in(\Z_4)^N$ is the sum of the Lee weights of its components. This weight function defines a metric in $(\Z_4)^N$, called the \emph{Lee distance}.
\par
Write $\C$ for the code $\C_f$ when $f$ is identically zero (in which case $f$ is planar). The dual code $\C^\perp$ is the $\Z_4$-Kerdock code described in~\cite{HamKumCalSloSol1994}. Let
\[
\phi:(\Z_4)^N\to(\F_2)^{2N}
\]
be the Gray map, which defines an isometry between $(\Z_4)^N$, equipped with the Lee distance, and $(\F_2)^{2N}$, equipped with the Hamming distance. Then, for $n\ge 3$ odd, $\phi(\C^\perp)$ is the classical Kerdock code and $\phi(\C)$ has the same parameters as the Preparata code (see~\cite{HamKumCalSloSol1994} for details on these codes).
\par
The Lee weight distribution of $\C^\perp$ has been determined in~\cite{HamKumCalSloSol1994}. The following more general result gives a characterisation of planar functions.
\begin{theorem}
\label{thm:C_weight_distribution}
The code $(\C_f)^\perp$ has the same Lee weight distribution as $\C^\perp$ if and only if $f$ is planar. In particular, if $f$ is planar, the Lee weight distribution of $(\C_f)^\perp$ is given in Table~\ref{tab:weight_distribution_n_odd} for odd $n$ and in Table~\ref{tab:weight_distribution_n_even} for even~$n$.
\begin{table}[t]
\centering
\caption{Weight distribution of $(\C_f)^\perp$ for odd $n$.}
\label{tab:weight_distribution_n_odd}
\begin{tabular}{c|c}
\hline
\text{weight} & \text{frequency}\\\hline\hline
$0$               & $1$\\
$2^n-2^{(n-1)/2}$ & $2^{n+1}(2^n-1)$\\
$2^n$             & $2^{n+2}-2$\\
$2^n+2^{(n-1)/2}$ & $2^{n+1}(2^n-1)$\\
$2^{n+1}$         & $1$\\\hline
\end{tabular}
\end{table}
\begin{table}[t]
\centering
\caption{Weight distribution of $(\C_f)^\perp$ for even $n$.}
\label{tab:weight_distribution_n_even}
\begin{tabular}{c|c}
\hline
\text{weight} & \text{frequency}\\\hline\hline
$0$           & $1$\\
$2^n-2^{n/2}$ & $2^n(2^n-1)$\\
$2^n$         & $2^{n+1}(2^n+1)-2$\\
$2^n+2^{n/2}$ & $2^n(2^n-1)$\\
$2^{n+1}$     & $1$\\\hline
\end{tabular}
\end{table}
\end{theorem}
\begin{proof}
Let $\omega$ be a primitive fourth root of unity. If $c=(c_1,\dots,c_N)$ is an element of $(\Z_4)^N$, then its Lee weight satisfies
\begin{equation}
\wt_L(c)=N-\re\bigg(\sum_{i=1}^N\omega^{c_i}\bigg).   \label{eqn:def_Lee_weight}
\end{equation}
Let $T:R_n\to\Z_4$ be the absolute trace function on $R_n$. We shall index elements of codewords by $\Gamma(R_n)$. For $a\in R_n$ and $b\in\Z_4$, consider the codeword
\[
c_{a,b}=\big(T\big(a(x+2\sqrt{f(x)})\big)+b\big)_{x\in\Gamma(R_n)}.
\]
By a folklore generalisation of Delsarte's Theorem~\cite[p.~208]{MacSlo1977} to codes over $\Z_4$, these are exactly the $4^{n+1}$ codewords of $(\C_f)^\perp$. From~\eqref{eqn:def_Lee_weight} we have
\begin{equation}
\wt_L(c_{a,b})=2^n-\re(\omega^b S_a),   \label{eqn:Lee_weight_cab}
\end{equation}
where
\[
S_a=\sum_{x\in\Gamma(R_n)}\omega^{T(a(x+2\sqrt{f(x)}))}.
\]
Since $z\mapsto \omega^{T(az)}$ are exactly the characters of the additive group of $R_n$, by Theorem~\ref{thm:diff_set} and~\eqref{eqn:diff_set_char}, the function $f$ is planar if and only if
\begin{equation}
\abs{S_a}^2=\begin{cases}
4^n & \text{for $a=0$}\\
0   & \text{for $a\in 2R_n\setminus\{0\}$}\\
2^n & \text{for $a\in R_n\setminus 2R_n$}.
\end{cases}   \label{eqn:characterisation_of_f}
\end{equation}
\par
Now let $f$ be planar. Using~\eqref{eqn:Lee_weight_cab}, we easily get the Lee weight distribution of the codewords $c_{a,b}$ when $a\in2R_n$ and $b\in\Z_4$. Next assume that $a\in R_n\setminus 2R_n$ and write $S_a=X+\omega Y$ for integers $X$ and $Y$. By Jacobi's two-square theorem, the only solutions to the Diophantine equation $X^2+Y^2=2^n$ are
\[
(X,Y)=\begin{cases}
(\pm2^{(n-1)/2},\pm2^{(n-1)/2})               & \text{for odd $n$}\\
(0,\pm2^{n/2})\;\;\text{or}\;\;(\pm2^{n/2},0) & \text{for even $n$}.
\end{cases}
\]
Therefore, for odd $n$, we have
\[
S_a=\pm2^{(n-1)/2}\pm2^{(n-1)/2}\omega.
\]
Hence, as $b$ ranges over $\Z_4$ and $a\in R_n\setminus2R_n$ is fixed, the expression $\re(\omega^bS_a)$ takes on each of the values $\pm 2^{(n-1)/2}$ twice. One can then get the Lee weight distribution from~\eqref{eqn:Lee_weight_cab}. Likewise, for even $n$, we have
\[
S_a=\pm2^{n/2}\quad\text{or}\quad \pm2^{n/2}\omega.
\]
Hence, as $b$ ranges over $\Z_4$ and $a\in R_n\setminus 2R_n$ is fixed, the expression $\re(\omega^bS_a)$ is zero twice and takes on each of the values $\pm 2^{n/2}$ once. The Lee weight distribution follows from~\eqref{eqn:Lee_weight_cab}.
\par
Now, if $f$ is not planar, then it follows easily from~\eqref{eqn:Lee_weight_cab} and the characterisation~\eqref{eqn:characterisation_of_f} of planar functions that the Lee weight distribution of $(\C_f)^\perp$ cannot coincide with that of $\C^\perp$.
\end{proof}
\par
For odd $n$, we have the following alternative characterisations of planar functions.
\begin{theorem}
\label{thm:Preparata}
For odd $n\ge 3$, the code $\C_f$ has minimum Lee distance $4$ or~$6$, where the value $6$ occurs if and only if $f$ is planar.
\end{theorem}
\begin{proof}
Recall that the \emph{type} of a codeword is defined as the enumerator of its nonzero entries. For example a codeword of type $1^22^4$ equals $1$ at two positions and equals $2$ at four positions.
\par
Notice that a nonzero codeword in $\C_f$ of Lee weight at most $3$ implies that there exists a codeword in $\C_f$ of type $2^1$, $2^2$, or $2^3$. Such codewords however cannot exist in $\C_f$ (for the same reason as the minimum distance of the extended Hamming code equals $4$). Hence the minimum Lee distance of $\C_f$ is at least $4$. 
\par
If $f$ is planar, the Lee weight distribution of $\C_f$ is independent of $f$ by Theorem~\ref{thm:C_weight_distribution} and a MacWilliams-type identity (see~\cite[\S~II.B]{HamKumCalSloSol1994}, for example). Hence, if $f$ is planar, the minimum Lee distance of $\C_f$ equals that of $\C$, which is $6$~\cite{HamKumCalSloSol1994}.
\par
We complete the proof by assuming that $f$ is not planar and show that $\C_f$ then contains a codeword of type $1^2(-1)^2$, and so has minimum distance at most $4$. The code $\C_f$ contains a codeword of type $1^2(-1)^2$ if and only if there exist distinct elements $u,v,x,y$ in $\Gamma(R_n)$ satisfying simultaneously the following two equations over $R_n$
\begin{gather*}
u+x=v+y\\
u+2\sqrt{f(u)}+x+2\sqrt{f(x)}=v+2\sqrt{f(v)}+y+2\sqrt{f(y)}.
\end{gather*}
By the definition~\eqref{eqn:def_oplus} of addition in $\Gamma(R_n)$, these equations are equivalent to the following two equations over $\Gamma(R_n)$
\begin{gather*}
u\oplus x=v\oplus y\\
ux\oplus f(u)\oplus f(x)=vy\oplus f(v)\oplus f(y).
\end{gather*}
From the first equation we infer that there exists $z\in\Gamma(R_n)$ such that $u=v\oplus z$ and $y=x\oplus z$. The second equation then becomes
\[
f(v)\oplus f(v\oplus z)\oplus vz=f(x)\oplus f(x\oplus z)\oplus xz.
\]
Since $f$ is not planar, this equation has a solution $(v,x,z)$, where $v$ and $x$ are distinct and $z\ne 0$. One then verifies that $u,v,x,y$ are also distinct.
\end{proof}
\par
A consequence of Theorem~\ref{thm:Preparata} is the following.
\begin{corollary}
\label{cor:Preparata}
For odd $n\ge 3$, the code $\phi(\C_f)$ punctured in one (arbitrary) coordinate has minimum distance $3$, $4$, or~$5$, where the value $5$ occurs if and only if $f$ is~planar.
\end{corollary}
\begin{proof}
The only part that is not immediate from Theorem~\ref{thm:Preparata} is that the code cannot have minimum distance $6$. But this value cannot occur since the code then violates a version of the Johnson bound~\cite{GoeSno1972}.
\end{proof}
\par
Let $\D_f$ be the code over $\F_2$ with parity check matrix~\eqref{eqn:pc_matrix_BCH}. If $f$ is almost perfect nonlinear, then $\D_f$ has parameters $(2^n-1,2^{2^n-2n-1},5)$ for $n\ge 4$. In contrast, by Corollary~\ref{cor:Preparata}, if $f$ is planar, then $\phi(\C_f)$ punctured in one coordinate has parameters $(2^n-1,2^{2^n-2n},5)$ for even $n\ge 4$, and so contains twice as many codewords as $\D_f$. If $f$ is planar, then $\phi(\C_f)$ punctured in one coordinate meets a version of the Johnson bound, and so is nearly perfect~\cite{GoeSno1972}.


\section{Planar monomial functions}
\label{sec:monomials}

We begin with providing a nontrivial example of planar monomial functions, in which
\[
\Tr_m(x)=x+x^2+\cdots+x^{2^{m-1}}
\]
denotes the trace function on $\F_{2^m}$.
\begin{theorem}
\label{thm:planar_monomial_1}
Let $c\in\F_{2^k}^*$ be such that $\Tr_k(c)=0$. Then the function
\[
x\mapsto cx^{2^k+1}
\]
is planar on $\F_{4^k}$.
\end{theorem}
\begin{proof}
We have to show that, for each $\e\in\F_{4^k}^*$, the mapping
\[
x\mapsto c(x+\e)^{2^k+1}+cx^{2^k+1}+\e x
\]
is a permutation of $\F_{4^k}$, or equivalently, the linear mapping
\begin{equation}
x\mapsto x^{2^k}\e+x\e^{2^k}+\e x/c   \label{eqn:linear_mapping}
\end{equation}
is a permutation of $\F_{4^k}$. This holds if the kernel of the mapping~\eqref{eqn:linear_mapping} is trivial. Hence, it is enough to show that
\[
x^{2^k-1}=\e^{2^k-1}+1/c
\]
has no solution $(x,\e)$ in $\F_{4^k}^*\times \F_{4^k}^*$. Let $\Gamma$ be the cyclic subgroup of $\F_{4^k}^*$ with order $2^k+1$. We show that
\[
\Gamma\cap (\Gamma+1/c)=\varnothing,
\]
which will prove the theorem.
\par
Let $y$ be in $\Gamma$. Then $y^{2^k+1}=1$ and, since $c\in\F_{2^k}^*$,
\begin{equation}
(y+1/c)^{2^k+1}=1+1/c^2+1/(cy)+y/c.   \label{eqn:y1c}
\end{equation}
Now suppose, for a contradiction, that $y$ also belongs to $\Gamma+1/c$. Then the left hand side of~\eqref{eqn:y1c} equals $1$, and thus
\begin{equation}
y^2+y/c+1=0.   \label{eqn:quadratic}
\end{equation}
We may set $z=yc$ to transform this quadratic equation into the standard form $z^2+z+c^2=0$, which has two solutions in $\F_{2^k}$ if and only if $\Tr_k(c^2)=0$~\cite[Ch.~9, Thm.~15]{MacSlo1977}. Since $\Tr_k(c)=0$ and $c\in\F_{2^k}$, we find that $y\in\F_{2^k}$. But $y$ is also in $\Gamma$, so that
\[
1=y^{2^k+1}=y^2,
\]
contradicting~\eqref{eqn:quadratic}.
\end{proof}
\par
We conjecture that the only planar exponents of $\F_{2^n}$ are the trivial examples $2^k$ and those identified in Theorem~\ref{thm:planar_monomial_1} and \cite[Theorem~1.1]{SchZie2013}.
\begin{conjecture}
\label{con:all_planar_monomials}
If $t$ is a planar exponent of $\F_{2^n}$, then $t$ is one of the values given in Table~\ref{tab:planar_exponents}.
\end{conjecture}
\par
The following partial answer to Conjecture~\ref{con:all_planar_monomials} is easy to prove.
\begin{proposition}
\label{thm:tminus2_coprime}
Let $t$ be an integer satisfying $\gcd(t-2,2^n-1)=1$. If $t$ is a planar exponent of $\F_{2^n}$, then $t$ is a power of $2$.
\end{proposition}
\begin{proof}
Suppose that $x\mapsto cx^t$ is planar on $\F_{2^n}$ for some $c\in\F_{2^n}^*$. Then
\[
x\mapsto c(x+\e)^t+cx^t+\e x
\]
is a permutation of $\F_{2^n}$ for each $\e\in\F_{2^n}^*$. Substituting $y=x/\e$, we see that
\[
y\mapsto (y+1)^t+y^t+(\e^{2-t}/c)y
\]
is a permutation of $\F_{2^n}$ for each $\e\in\F_{2^n}^*$. Hence, for each $\e\in\F_{2^n}^*$, the equation
\[
(y+1)^t+y^t+(z+1)^t+z^t=(\e^{2-t}/c)(y+z)
\]
has no solution $(x,y)$ in $\F_{2^n}\times\F_{2^n}$ satisfying $y\ne z$. Equivalently, writing
\[
D=\bigg\{\frac{(y+1)^t+y^t+(z+1)^t+z^t}{y+z}:y,z\in\F_{2^n},y\ne z\bigg\},
\]
we have $D\cap \{\e^{2-t}/c:\e\in\F_{2^n}^*\}=\varnothing$. But since $t-2$ is coprime to $2^n-1$, we have $\{\e^{2-t}/c:\e\in\F_{2^n}^*\}=\F_{2^n}^*$, hence $D=\{0\}$. Therefore, $(y+1)^t+y^t$ is constant for all $y\in\F_{2^n}$, which implies that $t$ is a power of two.
\end{proof} 
\begin{remark}
Proposition~\ref{thm:tminus2_coprime} corresponds to case (iv) of~\cite[Theorem~1.1]{BloBalBroStoSzo1999}.
\end{remark}
\par
We now focus on the relaxed problem of classifying the numbers that are planar exponents of $\F_{2^n}$ for infinitely many~$n$. The only known such numbers are the powers of $2$ and we have the following weaker form of Conjecture~\ref{con:all_planar_monomials}.
\begin{conjecture}
\label{con:all_exceptional_planar_monomials}
If $t$ is a planar exponent of $\F_{2^n}$ for infinitely many $n$, then~$t$ is a power of~$2$.
\end{conjecture}
\par
Our main result, Theorem~\ref{thm:exceptional_monomials_odd}, is a partial answer to this conjecture. This result will be proved in the remainder of this paper. The method is outlined below.
\par
Let $f:\F_{2^n}\to\F_{2^n}$ be of the form $f(x)=cx^t$ for some $c\in\F_{2^n}^*$ and let $\e\in\F_{2^n}^*$. Then the condition that~\eqref{eqn:planar_def} is a permutation is equivalent to the condition that the polynomial
\[
c(U+\e)^t+c(V+\e)^t+cU^t+cV^t+\e(U+V)
\]
has no zeros $(u,v)$ over $\F_{2^n}$ satisfying $u\ne v$. Substituting $U=\e X$ and $V=\e Y$, we see that this condition is in turn equivalent to the condition that the polynomial
\begin{equation}
(X+1)^t+(Y+1)^t+X^t+Y^t+a(X+Y)   \label{eqn:poly}
\end{equation}
has no zeros $(u,v)$ over $\F_{2^n}$ satisfying $u\ne v$, where $a=\e^{2-t}/c$. The polynomial~\eqref{eqn:poly} is divisible by $X+Y$. We are therefore interested in the zeros of the polynomial  
\begin{equation}
F_{t,a}(X,Y)=\frac{(X+1)^t+(Y+1)^t+X^t+Y^t+a(X+Y)}{X+Y}   \label{eqn:def_F}
\end{equation}
(which however could still have zeros on the line $X+Y$). We consider the affine plane curve defined by $F_{t,a}$ (and follow the usual convention to denote the curve and a defining polynomial by the same symbol). Then, defining a subset of $\F_{2^n}$ by
\begin{equation}
\A_n=\{\e^{2-t}/c:\e\in\F_{2^n}^*\},   \label{eqn:def_set_A}
\end{equation}
the function $x\mapsto cx^t$ is planar on $\F_{2^n}$ if and only if the curve $F_{t,a}$ has no rational points $(u,v)$ over $\F_{2^n}$ satisfying $u\ne v$ for some $a\in\A_n$. 
\par 
The number of rational points on a curve can be estimated using Weil's Theorem, which we quote in the following form (see~\cite[Theorem~5.4.1]{FriJar2008}, for example).
\begin{weil}
Let $F\in\F_q[X,Y]$ be an absolutely irreducible polynomial of degree $d$ and let $N$ be the number of rational points over $\F_q$ on the affine plane curve $F$. Then
\[
\abs{N-q-1}\le (d-1)(d-2)\sqrt{q}+d.
\]
\end{weil}
\par
A consequence of Weil's Theorem is the following.
\begin{proposition}
\label{pro:irreducibility_implies_rational_points}
If $F_{t,a}$ has an absolutely irreducible factor over $\F_{2^n}$ for some $a\ne 1$ in $\A_n$ and $n$ is sufficiently large, then $t$ is not a planar exponent of $\F_{2^n}$.
\end{proposition}
\begin{proof}
Let $a\in\A_n$ satisfy $a\ne 1$ and suppose that $F_{t,a}$ has an absolutely irreducible factor over $\F_{2^n}$. By the above discussion, it is sufficient to show that, if $n$ is sufficiently large, then the curve $F_{t,a}$ has rational points $(u,v)$ over $\F_{2^n}$ satisfying $u\ne v$. Since the degree of $F_{t,a}$ is at most $t-2$, by Weil's Theorem the number of rational points over $\F_{2^n}$ on the curve $F_{t,a}$ is at least
\[
2^n-(t-3)(t-4)2^{n/2}-t+3.
\]
By taking partial derivatives of the numerator of $F_{t,a}$, we see that $F_{t,a}$ is never divisible by $X+Y$ since $a\ne 1$ (this fails for $t=2^s+1$ with $s>0$ if we allow $a=1$). Hence, if $(u,u)$ is on the curve $F_{t,a}$, then $u$ is a root of a nonzero polynomial of bounded degree. Therefore, if $n$ is sufficiently large, the curve $F_{t,a}$ has rational points $(u,v)$ over $\F_{2^n}$ satisfying $u\ne v$.
\end{proof}
\par
In view of Proposition~\ref{pro:irreducibility_implies_rational_points}, Conjecture~\ref{con:all_exceptional_planar_monomials} is proved by showing that, when $t$ is not a power of $2$, $F_{t,a}$ has an absolutely irreducible factor over $\F_{2^n}$ for some $a\in\A_n$ satisfying $a\ne 1$ and all sufficiently large~$n$.
\par
The following corollary to Lucas's Theorem will be useful.
\begin{lemma}
\label{lem:Lucas}
The binomial coefficient $\tbinom{m}{k}$ is even if and only if at least one of the base-$2$ digits of $k$ is greater than the corresponding digit of $m$.
\end{lemma}
\par
Instead of looking at $F_{t,a}$ directly, we consider its homogenised version $H_{t,a}(X,Y,Z)$. If $t$ is not a power of two, we find from Lemma~\ref{lem:Lucas} that
\begin{equation}
H_{t,a}(X,Y,Z)=\frac{(X+Z)^t+(Y+Z)^t+X^t+Y^t+a(X+Y)Z^{t-1}}{Z^{2^j}(X+Y)},   \label{eqn:def_H}
\end{equation}
where $j$ is the largest power of $2$ that divides $t$. Of course, $F_{t,a}$ has an absolutely irreducible factor if and only if $H_{t,a}$ has an absolutely irreducible factor. Our strategy is to consider the projective plane curve defined by $H_{t,a}$ over the algebraic closure $\F$ of $\F_2$ and derive a contradiction to Bezout's Theorem (see~\cite[\S~5.3]{Ful2008}, for example) under the assumption that $H_{t,a}$ has no absolutely irreducible factor over $\F_{2^n}$.
\begin{bezout}
Let $A$ and $B$ be two projective plane curves over an algebraically closed field $\K$, having no component in common. Then
\[
\sum_PI_P(A,B)=(\deg A)(\deg B),
\]
where the sum runs over all points in the projective plane $\P^2(\K)$.
\end{bezout}
\par
Notice that $I_P(A,B)$ is the \emph{intersection number} of $A$ and $B$ at $P$, whose precise definition is neither recalled nor required in this paper. We shall rather use some properties of the intersection number, which allows us to compute it in certain cases of interest. In Section~\ref{sec:intersection_numbers}, we shall obtain general upper bounds on the intersection number $I_P(A,B)$, where $F_{t,a}=AB$ is an arbitrary factorisation of $F_{t,a}$ and $P$ is a point in the plane $\P^2(\F)$. The desired contradiction to Bezout's Theorem is then derived in Section~\ref{sec:main_proof}.


\section{Computation of intersection numbers}
\label{sec:intersection_numbers}

\subsection{Some results on intersection numbers}

Let $F$ be an affine plane curve (which we always assume to be defined over an algebraically closed field), let $P=(u,v)$ be a point in the plane, and write
\[
F(X+u,Y+v)=F_0(X,Y)+F_1(X,Y)+F_2(X,Y)+\cdots,
\]
where $F_i$ is either zero or homogeneous of degree $i$. The \emph{multiplicity} of $F$ at $P$, written as $m_P(F)$, is the smallest integer $m$ such that $F_m\ne 0$ and $F_i=0$ for $i<m$; the polynomial is $F_m$ is the \emph{tangent cone} of $F$ at $P$. A divisor of the tangent cone is called a \emph{tangent} of $F$ at $P$. The point $P$ is on the curve $F$ if and only if $m_P(F)\ge 1$. If $P$ is on $F$, then $P$ is a \emph{simple} point of $F$ if $m_P(F)=1$, otherwise $P$ is a \emph{singular} point of $F$.
\par
Now let $F^*(X,Y,Z)$ be the homogenised polynomial of $F(X,Y)$ and write $P^*=(u,v,1)$ (in homogeneous coordinates). Then the multiplicity of the projective plane curve $F^*$ at $P^*$, also written as $m_{P^*}(F^*)$, is by definition $m_P(F)$. Likewise the intersection number $I_{P^*}(A^*,B^*)$ is by definition $I_P(A,B)$, where $A^*$ and $B^*$ are the homogenised polynomials of $A$ and $B$, respectively (see~\cite[Ch.~5]{Ful2008} for details). We may therefore restrict our analysis to affine plane curves.
\par
One important property of the intersection number is that $I_P(A,B)=0$ if $P$ is not a singular point of $AB$. This is a special case of the following more general property.
\begin{lemma}[{\cite[Ch.~3, Property (5)]{Ful2008}}]
\label{lem:intersection_number_and_multiplicity}
Let $A$ and $B$ be two affine plane curves and suppose that the tangent cones of $A$ and $B$ do not share a common factor. Let $P$ be a point in the plane. Then $I_P(A,B)=m_P(A)m_P(B)$.
\end{lemma}
\par
It is an easy exercise to obtain the following result as a corollary of Lemma~\ref{lem:intersection_number_and_multiplicity} (see Janwa, McGuire, and Wilson~\cite[Proposition 2]{JanMcGWil1995}).
\begin{corollary}
\label{cor:intersection_number_m_m1_coprime}
Let $F$ be an affine plane curve and suppose that $F=AB$. Let $P=(u,v)$ be a point in the plane and write
\[
F(X+u,Y+v)=F_m(X,Y)+F_{m+1}(X,Y)+\cdots,
\]
where $F_i$ is zero or homogeneous of degree $i$ and $F_m\ne 0$. Let~$L$ be a linear polynomial and suppose that $F_m=L^m$ and $L\nmid F_{m+1}$. Then $I_P(A,B)=0$.
\end{corollary}
\par
We shall require one further result to compute intersection numbers, whose proof idea follows that of~\cite[Lemma~8]{HerMcG2012}.
\begin{lemma}
\label{lem:intersection_number_linear_term}
Let $F$ be an affine plane curve over a field of characteristic~two and suppose that $F=AB.$ Let $P=(u,v)$ be a point in the plane and write
\[
F(X+u,Y+v)=F_m(X,Y)+F_{m+1}(X,Y)+\cdots,
\]
where $F_i$ is zero or homogeneous of degree $i$ and $F_m\ne 0$. Let~$L$ be a linear polynomial and suppose that $F_m=L^m$ and $L\parallel F_{m+1}$. Then $I_P(A,B)=0$ or~$m$.
\end{lemma}
\begin{proof}
Write
\begin{align*}
A(X+u,Y+v)&=A_r(X,Y)+A_{r+1}(X,Y)+\cdots
\intertext{and}
B(X+u,Y+v)&=B_s(X,Y)+B_{s+1}(X,Y)+\cdots,
\end{align*}
where $A_i$ and $B_i$ are zero or homogeneous of degree $i$ and $A_r$ and $B_s$ are nonzero. Since $F_m=L^m$, we have, up to constant factors, $A_r=L^j$ and $B_s=L^{m-j}$ for some $j\in\{0,\dots,m\}$. Also, 
\begin{equation}
F_{m+1}=A_rB_{s+1}+A_{r+1}B_s,   \label{eqn:F_m1_AB}
\end{equation}
and since $L\parallel F_{m+1}$, we find that $\gcd(A_r,B_s)=1$ or~$L$. If $\gcd(A_r,B_s)=1$, then either $m_P(A)=0$ or $m_P(B)=0$ and $I_P(A,B)=0$ by Lemma~\ref{lem:intersection_number_and_multiplicity}.
\par
Now suppose that $\gcd(A_r,B_s)=L$, which implies that $m\ge 2$. Without loss of generality, we may assume that $A_r=L$ and $B_s=L^{m-1}$, so that $r=1$ and $s=m-1$. Define
\[
C(X,Y)=A(X,Y)L(X-u,Y-v)^{m-2}+B(X,Y).
\]
Then, by a general property of intersection numbers~\cite[Ch.~3, Property~(7)]{Ful2008}, we find that
\[
I_P(A,B)=I_P(A,C).
\]
We have
\[
C(X+u,Y+v)=A_2(X,Y)L(X,Y)^{m-2}+B_m(X,Y)+\text{higher order terms}.
\]
If $m=2$, then it follows from $L\parallel F_{m+1}$ and~\eqref{eqn:F_m1_AB} that $L\nmid A_2+B_2$. If $m>2$, we find from~\eqref{eqn:F_m1_AB} that $L\nmid B_m$. In either case, the tangent cones of $A$ and $C$ do not share a common factor and therefore $I_P(A,C)=m_P(A)m_P(C)$ by Lemma~\ref{lem:intersection_number_and_multiplicity}. This completes the proof since $m_P(A)=1$ and $m_P(C)=m$.
\end{proof}


\subsection{Singular points at infinity of $H_{t,a}$}

We now study the intersection numbers $I_P(A,B)$, where $H_{t,a}=AB$ is some factorisation and $P$ is a singular point at infinity of $H_{t,a}$, namely a point of the form $(u,v,0)$. Since $H_{t,a}$ is symmetric in $X$ and $Y$, we can assume that $v=1$. It is then sufficient to consider the dehomogenisation
\[
G_{t,a}(X,Z)=H_{t,a}(X,1,Z),
\]
so that
\[
G_{t,a}(X,Z)=\frac{(X+Z)^t+(Z+1)^t+X^t+a(X+1)Z^{t-1}+1}{Z(X+1)}.
\]
The result of this section is the following.
\begin{lemma}
\label{lem:I_infinity}
Let $t$ be a number of the form $2^k\ell+1$ for integers $k\ge 1$ and odd $\ell\ge 3$. Let $P=(u,0)$ be a singular point of $G_{t,a}$ and suppose that $G_{t,a}=AB$ is a factorisation over $\F$. Then $I_P(A,B)\le 4^{k-1}$.
\end{lemma}
\begin{proof}
Write $\widetilde G_{t,a}$ for the numerator of $G_{t,a}$, namely
\begin{equation}
\widetilde G_{t,a}(X,Z)=Z(X+1)G_{t,a}(X,Z).   \label{eqn:def_numerator_G}
\end{equation}
Next we compute the multiplicities of $G_{t,a}$ and $\widetilde G_{t,a}$ at $P$. Write
\begin{align*}
G_{t,a}(X+u,Z)&=G_0(X,Z)+G_1(X,Z)+G_2(X,Z)+\cdots\\
\intertext{and}
\widetilde G_{t,a}(X+u,Z)&=\widetilde G_0(X,Z)+\widetilde G_1(X,Z)+\widetilde G_2(X,Z)+\cdots,
\end{align*}
where $G_i$ and $\widetilde G_i$ are either zero or homogeneous of degree $i$. From~\eqref{eqn:def_numerator_G} we find that
\begin{equation}
\widetilde G_i(X,Z)=XZG_{i-2}(X,Z)+Z(u+1)G_{i-1}(X,Z),   \label{eqn:GiFi_infinity}
\end{equation}
where, by convention, $G_{-1}=G_{-2}=0$. We have
\[
\widetilde G_{t,a}(X+u,Z)=\sum_{j=0}^{t}\binom{t}{j}\Big[u^{t-j}((X+Z)^j+X^j)+Z^j\Big]+a(X+u+1)Z^{t-1}+1.
\]
Since $P$ is a singular point of $G_{t,a}$, and so is a singular point of $\widetilde G_{t,a}$, we have $\widetilde G_0=\widetilde G_1=0$. From Lemma~\ref{lem:Lucas} we see that $\widetilde G_i=0$ for each $i\in\{2,\dots,2^k-1\}$. Furthermore, since $\ell\ge 3$,
\begin{equation}
\widetilde G_{2^k}(X,Z)=(u^{t-2^k}+1)Z^{2^k}   \label{eqn:H2k_infinity}
\end{equation}
and
\[
\widetilde G_{2^k+1}(X,Z)=u^{t-2^k-1}((X+Z)^{2^k+1}+X^{2^k+1})+Z^{2^k+1}.
\]
We now see that the multiplicity of $\widetilde G_{t,a}$ at $P=(1,0)$ is $2^k+1$, while  the multiplicity of $\widetilde G_{t,a}$ at $P=(u,0)$ for $u\ne 1$ can be either $2^k$ or $2^k+1$. Using~\eqref{eqn:GiFi_infinity}, it is then straightforward to work out the corresponding multiplicities of $G_{t,a}$. The results are summarised in Table~\ref{tab:sing_points_mult}.
\begin{table}[ht]
\centering
\caption{Multiplicities of $G_{t,a}$ and $\widetilde G_{t,a}$ at their singular points.}
\begin{tabular}{c|c|c|c}
\hline
Type  & Point $P$           & $m_P(\widetilde G_{t,a})$ & $m_P(G_{t,a})$\\ \hline\hline
A  & $(1,0)$           & $2^k+1$                   & $2^k-1$\\
B  & $(u,0)$, $u\ne 1$ & $2^k+1$                   & $2^k$\\
C  & $(u,0)$, $u\ne 1$ & $2^k$                     & $2^k-1$\\\hline
\end{tabular}
\label{tab:sing_points_mult}
\end{table}
\par
We shall need the following observation, which will be proved at the end of this section.
\begin{claim}
\label{cla:H2k1_splits_infinity}
$\widetilde G_{2^k+1}$ splits into $2^k+1$ distinct factors over its splitting field.
\end{claim}
\par
We resume the proof of Lemma~\ref{lem:I_infinity} and distinguish three cases for $P$, according to Table~\ref{tab:sing_points_mult}.
\begin{list}{$\bullet$}{\leftmargin=1.5em \itemindent=0em \itemsep=-0.3ex}
\item {\itshape $P$ is a point of type A.} In this case, the multiplicity of $G_{t,a}$ at $P$ is $2^k-1$ and from~\eqref{eqn:GiFi_infinity} we have 
\[
\widetilde G_{2^k+1}(X,Z)=XZG_{2^k-1}(X,Z).
\]
Therefore, by Claim~\ref{cla:H2k1_splits_infinity}, $G_{2^k-1}$, the tangent cone of $G_{t,a}$ at $P$, has no multiple factors over its splitting field. Lemma~\ref{lem:intersection_number_and_multiplicity} then implies $I_P(A,B)=m_P(A)m_P(B)$.

\item {\itshape $P$ is a point of type B.} In this case, the multiplicity of $G_{t,a}$ at $P$ is $2^k$ and from~\eqref{eqn:GiFi_infinity} we have
\[
\widetilde G_{2^k+1}(X,Z)=Z(u+1)G_{2^k}(X,Z).
\]
Thus by Claim~\ref{cla:H2k1_splits_infinity}, $G_{2^k}$, the tangent cone of $G_{t,a}$ at $P$, has no multiple factors over its splitting field and so Lemma~\ref{lem:intersection_number_and_multiplicity} gives $I_P(A,B)=m_P(A)m_P(B)$.

\item {\itshape $P$ is a point of type C.} Now the multiplicity of $G_{t,a}$ at $P$ is $2^k-1$. From~\eqref{eqn:GiFi_infinity} we find that
\begin{align*}
\widetilde G_{2^k}(X,Z)&=Z(u+1)G_{2^k-1}(X,Z)\\
\widetilde G_{2^k+1}(X,Z)&=XZG_{2^k-1}(X,Z)+Z(u+1)G_{2^k}(X,Z).
\end{align*} 
From~\eqref{eqn:H2k_infinity} we see that the tangent cone of $G_{t,a}$ at $P$ equals
\[
G_{2^k-1}(X,Z)=\frac{u^{t-2^k}+1}{u+1}Z^{2^k-1}
\]
and then, by Claim~\ref{cla:H2k1_splits_infinity}, $Z\nmid G_{2^k}$. Thus $I_P(A,B)=0$ by Corollary~\ref{cor:intersection_number_m_m1_coprime}.
\end{list}
Now from the three cases above we conclude that $I_P(A,B)$ equals either zero or $m_P(A)m_P(B)$. But since
\[
m_P(A)+m_P(B)=m_P(G_{t,a})\le 2^k,
\]
we find that $I_P(A,B)\le (2^{k-1})^2$, as required.
\end{proof}
\par
It remains to prove the claim invoked in the proof of Lemma~\ref{lem:I_infinity}.
\begin{proof}[Proof of Claim~\ref{cla:H2k1_splits_infinity}]
We show that
\begin{align}
\gcd(\widetilde G_{2^k+1},\partial \widetilde G_{2^k+1}/\partial X) & \in\F[Z]
\label{eqn:gcd_X}
\intertext{and}
\gcd(\widetilde G_{2^k+1},\partial \widetilde G_{2^k+1}/\partial Z) & \in\F[X].   \label{eqn:gcd_Z}
\end{align}
The assertion~\eqref{eqn:gcd_X} follows since
\[
\partial \widetilde G_{2^k+1}/\partial X=(u^{\ell-1}Z)^{2^k}.
\]
To prove~\eqref{eqn:gcd_Z}, first observe that $P=(0,0)$ is not a singular point of $\widetilde G_{t,a}$ since then $\widetilde G_1(X,Z)=Z$, and so it is not a singular point of $G_{t,a}$. Hence we may assume that $u\ne 0$. We have
\[
\partial \widetilde G_{2^k+1}/\partial Z=(u^{\ell-1}X+(u^{\ell-1}+1)Z)^{2^k}.
\]
Hence $\partial \widetilde G_{2^k+1}/\partial Z$ has only one factor, namely
\begin{equation}
X+\frac{u^{\ell-1}+1}{u^{\ell-1}}Z.   \label{eqn:factor_G2k1}
\end{equation}
We readily verify that
\[
\widetilde G_{2^k+1}\bigg(\frac{u^{\ell-1}+1}{u^{\ell-1}}Z,Z\bigg)=u^{(\ell-1)(2^k-1)}(u^{\ell-1}+1)Z^{2^k+1}.
\]
Hence, since $u\ne 0$,~\eqref{eqn:factor_G2k1} divides $\widetilde G_{2^k+1}$ only if $u=1$. However, for $u=1$,~\eqref{eqn:factor_G2k1} equals $X$, which proves~\eqref{eqn:gcd_Z}.
\end{proof}


\subsection{Affine singular points of $H_{t,a}$}

We are now interested in the intersection numbers $I_P(A,B)$, where $H_{t,a}=AB$ and $P$ is an affine singular point of $H_{t,a}$, namely $P$ is of the form $(u,v,1)$. We work with the dehomogenisation
\[
F_{t,a}(X,Y)=H_{t,a}(X,Y,1),
\]
as given in~\eqref{eqn:def_F}. Let $\widetilde F_{t,a}(X,Y)$ be the numerator of $F_{t,a}(X,Y)$, so that
\[
\widetilde F_{t,a}(X,Y)=(X+1)^t+(Y+1)^t+X^t+Y^t+a(X+Y).
\]
\par
Our analysis crucially relies on restricting $a$ to values in a subset of $\A_n$, which we define next.
\begin{definition}
\label{def:B}
Let $\B_n$ be the set of all $a\in\A_n$ such that all singular points $(u,v)$ of $\widetilde F_{t,a}$ satisfy each of
\begin{align*}
(u+1)^{t-2^k}&\ne u^{t-2^k}\\
(u+1)^{t-2^k-1}&\ne u^{t-2^k-1}\\
(v+1)^{t-2^k-1}&\ne v^{t-2^k-1}.
\end{align*}
\end{definition}
\par
\begin{lemma}
\label{lem:size_of_B}
The set $\B_n$ contains an element not equal to $1$ for all sufficiently large~$n$.
\end{lemma}
\begin{proof}
Let $\Pc$ be the set of points $(u,v)\in\F\times\F$ that satisfy at least one of
\begin{align*}
(u+1)^{t-2^k}+u^{t-2^k}&=0\\
(u+1)^{t-2^k-1}+u^{t-2^k-1}&=0\\
(v+1)^{t-2^k-1}+v^{t-2^k-1}&=0.
\end{align*}
Since $t-2^k$ is constant, we find by a degree argument that $\Pc$ has finite size. Then, by Definition~\ref{def:B}, $a\in\A_n$ belongs to $\B_n$ if no point in $\Pc$ is a singular point of $\widetilde F_{t,a}$. By looking at the homogeneous part of degree $1$ of $\widetilde F_{t,a}(X+u,Y+v)$, we see that a necessary condition for $(u,v)$ to be a singular point of $\widetilde F_{t,a}$ is 
\begin{equation}
(u+1)^{t-1}+u^{t-1}=a.   \label{eqn:cond_sing_point_u}
\end{equation}
But from the definition~\eqref{eqn:def_set_A} of $\A_n$ we have
\[
\abs{\A_n}=\frac{2^n-1}{\gcd(2^n-1,t-2)}\ge \frac{2^n-1}{t-2},
\]
and so, for all sufficiently large $n$, we can choose an $a\ne 1$ in $\A_n$ such that~\eqref{eqn:cond_sing_point_u} is not satisfied for each $(u,v)\in\Pc$. This $a\in\A_n$ belongs to $\B_n$ since none of the points in $\Pc$ is a singular point of $\widetilde F_{t,a}$.
\end{proof}
\par
We now state the main result of this section.
\begin{lemma}
\label{lem:I_affine}
Let $t$ be a number of the form $2^k\ell+1$ for integers $k\ge 1$ and odd $\ell\ge 1$ and let $a\in\B_n$. Suppose that $F_{t,a}=AB$ is a factorisation over~$\F$ and let $P$ be a singular point of $F_{t,a}$.
\begin{enumerate}[(i)]
\item If $P=(u,u)$, then $m_P(F_{t,a})=2^k-1$ and $I_P(A,B)=0$.
\item If $P=(u,v)$ with $u\ne v$, then $m_P(F_{t,a})=2^k$ and $I_P(A,B)=2^k$.
\end{enumerate}
\end{lemma}
\begin{proof}
We shall first compute the multiplicities of $F_{t,a}$ and $\widetilde F_{t,a}$ at $P=(u,v)$. Write
\begin{align*}
F_{t,a}(X+u,Y+v)&=F_0(X,Y)+F_1(X,Y)+F_2(X,Y)+\cdots\\
\intertext{and}
\widetilde F_{t,a}(X+u,Y+v)&=\widetilde F_0(X,Y)+\widetilde F_1(X,Y)+\widetilde F_2(X,Y)+\cdots,
\end{align*}
where $F_i$ and $\widetilde F_i$ are either zero or homogeneous of degree $i$. We have
\begin{multline}
\widetilde F_{t,a}(X+u,Y+v)=a(X+Y+u+v)\\
+\sum_{j=0}^t\binom{t}{j}\big(\big[(u+1)^{t-j}+u^{t-j}\big]X^j+\big[(v+1)^{t-j}+v^{t-j}\big]Y^j\big).   \label{eqn:F_expanded}
\end{multline}
Since $P$ is a singular point of $F_{t,a}$, and so is a singular point of $\widetilde F_{t,a}$, we have $\widetilde F_0=\widetilde F_1=0$. From Lemma~\ref{lem:Lucas} we see that $\widetilde F_i=0$ for each $i\in\{2,\dots,2^k-1\}$. Furthermore, 
\begin{equation}
\widetilde F_{2^k}(X,Y)=((u+1)^{t-2^k}+u^{t-2^k})X^{2^k}+((v+1)^{t-2^k}+v^{t-2^k})Y^{2^k}.   \label{eqn:H2k_affine}
\end{equation}
Since $a\in\B_n$, we see from Definition~\ref{def:B} that $\widetilde F_{2^k}$ is never zero and so
\begin{equation}
m_P(\widetilde F_{t,a})=2^k.   \label{eqn:mult_F_twiddle}
\end{equation}
To compute the multiplicity of $F_{t,a}$ at $P$, we use
\begin{equation}
\widetilde F_i(X,Y)=(X+Y)F_{i-1}(X,Y)+(u+v)F_i(X,Y),   \label{eqn:HiGi_affine}
\end{equation}
where, by convention, $F_{-1}=0$. We now prove the two cases of the lemma separately, using the following claim proved at the end of this section.
\begin{claim}
\label{cla:H2k1_splits_affine}
$\widetilde F_{2^k+1}$ splits into $2^k+1$ distinct factors over its splitting field.
\end{claim}
\par
\begin{list}{$\bullet$}{\leftmargin=1.5em \itemindent=0em \itemsep=-0.3ex}
\item {\itshape $P=(u,u)$.} In this case, we have $m_P(F_{t,a})=2^k-1$ by~\eqref{eqn:mult_F_twiddle} and~\eqref{eqn:HiGi_affine}. Furthermore, from~\eqref{eqn:HiGi_affine},
\begin{align*}
\widetilde F_{2^k}(X,Y)&=(X+Y)F_{2^k-1}(X,Y)\\
\widetilde F_{2^k+1}(X,Y)&=(X+Y)F_{2^k}(X,Y),
\end{align*}
and then from~\eqref{eqn:H2k_affine},
\[
F_{2^k-1}(X,Y)=((u+1)^{t-2^k}+u^{t-2^k})(X+Y)^{2^k-1}.
\]
By Claim~\ref{cla:H2k1_splits_affine}, $\widetilde F_{2^k+1}$ has no multiple factors over its splitting field, and so $X+Y$ does not divide $F_{2^k}$. Thus $I_P(A,B)=0$ by Corollary~\ref{cor:intersection_number_m_m1_coprime}.\\

\item {\itshape $P=(u,v)$ with $u\ne v$.} In this case, we have $m_P(F_{t,a})=2^k$ by~\eqref{eqn:mult_F_twiddle} and~\eqref{eqn:HiGi_affine}. From~\eqref{eqn:HiGi_affine} we have
\begin{align*}
\widetilde F_{2^k}&=(u+v)F_{2^k}\\
\widetilde F_{2^k+1}&=(X+Y)F_{2^k}+(u+v)F_{2^k+1}.
\end{align*}
Since $\widetilde F_{2^k+1}$ has no multiple factors by Claim~\ref{cla:H2k1_splits_affine}, we conclude that $F_{2^k}$ and $F_{2^k+1}$ share at most one factor. Furthermore, from~\eqref{eqn:H2k_affine}, we see that
\[
F_{2^k}(X,Y)=(a_1X+a_2Y)^{2^k}\quad\text{for some $a_1,a_2\in\F$}.
\]
If $a_1X+a_2Y$ does not divide $F_{2^k+1}$, then $I_P(A,B)=0$ by Corollary~\ref{cor:intersection_number_m_m1_coprime}, so assume that $F_{2^k}$ and $F_{2^k+1}$ share the factor $a_1X+a_2Y$. This factor must divide $F_{2^k+1}$ exactly and thus $I_P(A,B)=0$ or $2^k$ by Lemma~\ref{lem:intersection_number_linear_term}.
\end{list}
This completes the proof.
\end{proof}
\par
We now prove the claim invoked in the proof of Lemma~\ref{lem:I_affine}.
\begin{proof}[Proof of Claim~\ref{cla:H2k1_splits_affine}]
From~\eqref{eqn:F_expanded} we find that $\widetilde F_{2^k+1}(X,Y)$ equals
\[
((u+1)^{t-2^k-1}+u^{t-2^k-1})X^{2^k+1}+((v+1)^{t-2^k-1}+v^{t-2^k-1})Y^{2^k+1}.
\]
Since $a\in\B_n$, we readily verify with Definition~\ref{def:B} that
\[
\gcd(\widetilde F_{2^k+1},\partial\widetilde F_{2^k+1}/\partial X)=\gcd(\widetilde F_{2^k+1},\partial\widetilde F_{2^k+1}/\partial Y)=1.
\]
This proves the claim.
\end{proof}


\section{Proof of Theorem~\ref{thm:exceptional_monomials_odd}}
\label{sec:main_proof}

Let $t>1$ be an odd integer. Recall that, in view of Proposition~\ref{pro:irreducibility_implies_rational_points}, we wish to show that $F_{t,a}$, given in~\eqref{eqn:def_F} (or equivalently $H_{t,a}$, given in~\eqref{eqn:def_H}) has an absolutely irreducible factor over $\F_{2^n}$ for some $a\ne 1$ in $\A_n$ and for all sufficiently large $n$. 
\par
The case that $t=2^k+1$ is particularly easy to~handle.
\begin{proposition}
\label{pro:abs_ireducible_l1}
Let $t$ be a number of the form $2^k+1$ for integral $k\ge 1$. Then $F_{t,a}$ has an absolutely irreducible factor for some $a\ne 1$ in $\A_n$ and for all sufficiently large $n$.
\end{proposition}
\begin{proof}
Notice that $F_{t,a}$ simplifies to
\[
F_{t,a}(X,Y)=(X+Y)^{2^k-1}+a+1.
\]
We claim that, for all sufficiently large $n$, we can choose $a\ne 1$ in $\A_n$ such that
\[
a+1=b^{2^k-1}.
\] 
for some $b\in\F_{2^n}^*$. This will prove the proposition since then $X+Y+b$ divides $F_{t,a}$. By the definition~\eqref{eqn:def_set_A} of $\A_n$, the claim is equivalent to the existence of $\e,b\in\F_{2^n}^*$ such that, for all $c\in\F_{2^n}^*$,
\begin{equation}
\e^{1-2^k}/c+1=b^{2^k-1},   \label{eqn:eps_b}
\end{equation}
which in turn is equivalent to
\begin{equation}
\e^{2^k-1}+x^{2^k-1}=1/c,   \label{eqn:diag_eqn}
\end{equation}
where $x=\e b$. It is well known~\cite[Example~ 6.38]{LidNie1997} that the number of solutions $(\e,x)\in\F_{2^n}\times\F_{2^n}$ to the equation~\eqref{eqn:diag_eqn} is at least
\[
2^n-(2^k-2)(2^k-3)2^{n/2}-2^k+2.
\]
Since there are at most $2^k-1$ solutions of the form $(0,x)$ and at most $2^k-1$ solutions of the form $(\e,0)$, we find that, for all sufficiently large $n$, there exist $\e,x\in\F_{2^n}^*$ satisfying~\eqref{eqn:diag_eqn}. Hence, for all sufficiently large $n$, there exist $\e,b\in\F_{2^n}^*$ satisfying~\eqref{eqn:eps_b}, as required.
\end{proof}
\par
Henceforth, we assume that $t=2^k\ell+1$ for integers $k\ge 1$ and odd $\ell\ge 3$. We shall factor $H_{t,a}$ into putative factors $A$ and $B$ over some extension of $\F_{2^n}$ and derive a contradiction to Bezout's theorem, using our estimates for $I_P(A,B)$. Since $I_P(A,B)=0$ if $P$ is a simple point of $AB$, the sum in Bezout's Theorem can be taken over the singular points of~$AB$. The main results of Section~\ref{sec:intersection_numbers} can be restated as follows ((i) follows from Lemma~\ref{lem:I_infinity} and the remarks preceding it and (ii) and (iii) follow from Lemmas~\ref{lem:size_of_B} and~\ref{lem:I_affine}).
\begin{corollary}
\label{cor:sing_points}
Let $t$ be a number of the form $2^k\ell+1$ for integers $k\ge 1$ and odd $\ell\ge 3$. Let $P$ be a singular point of $H_{t,a}$ and suppose that $H_{t,a}=AB$ is a factorisation of $H_{t,a}$ over $\F$. Then, for some $a\ne 1$ in $\A_n$ and all sufficiently large $n$, the following holds:
\begin{enumerate}[(i)]
\item If $P=(u,v,0)$, then $I_P(A,B)\le 4^{k-1}$.
\item If $P=(u,u,1)$, then  $m_P(H_{t,a})=2^k-1$ and $I_P(A,B)=0$.
\item If $P=(u,v,1)$ and $u\ne v$, then $m_P(H_{t,a})=2^k$ and $I_P(A,B)\le 2^k$.
\end{enumerate}
\end{corollary}
\par
It remains to count the number of singular points of $H_{t,a}$. To do so, we consider the numerator of $H_{t,a}$, namely
\[
\widetilde H_{t,a}(X,Y,Z)=(X+Z)^t+(Y+Z)^t+X^t+Y^t+a(X+Y)Z^{t-1}.
\] 
Recall that a point $P$ on a projective plane curve defined by $H(X,Y,Z)$ is a singular point of $H$ if and only if the partial derivatives of $H$ with respect to $X$, $Y$, and $Z$ vanish at~$P$. Since $t$ is odd, we have 
\begin{align*}
\partial\widetilde H_{t,a}/\partial X&=(X+Z)^{t-1}+X^{t-1}+aZ^{t-1}\\
\partial\widetilde H_{t,a}/\partial Y&=(Y+Z)^{t-1}+Y^{t-1}+aZ^{t-1}\\
\partial\widetilde H_{t,a}/\partial Z&=(X+Z)^{t-1}+(Y+Z)^{t-1}.
\end{align*}
Recalling that $t=2^k\ell+1$ and $\ell\ge 3$, it is then readily verified that the possible singular points of $\widetilde H_{t,a}$ are of one of the following types:
\begin{list}{$\bullet$}{\leftmargin=1em \itemindent=0em \itemsep=0ex}
\item {\itshape Points at infinity:} $(u,1,0)$ satisfying $u^\ell=1$,
\item {\itshape Affine points:} $(u,v,1)$ satisfying
\begin{equation}
\left\{
\begin{aligned}
(u+1)^\ell&=u^\ell+a^{2^{-k}}\\
(v+1)^\ell&=v^\ell+a^{2^{-k}}\\
(u+1)^\ell&=(v+1)^\ell.
\end{aligned}   \label{eqn:cond_affine_sing_point}
\right.
\end{equation}
\end{list}
\par
\begin{lemma}
\label{lem:number_of_sing_points}
Let $t$ be a number of the form $2^k\ell+1$ for integers $k\ge 1$ and odd $\ell\ge 3$. Then, for each nonzero $a\in\F$, the curve $H_{t,a}$ has at most $\ell$ singular points at infinity and at most $(\ell-1)(\ell-2)/2$ affine singular points $(u,v,1)$ satisfying $u\ne v$.
\end{lemma}
\begin{proof}
First observe that each singular point of $H_{t,a}$ is also a singular point of $\widetilde H_{t,a}$. It is readily verified that $\widetilde H_{t,a}$ has at most $\ell$ singular points at infinity, thus $H_{t,a}$ has at most $\ell$ such singular~points.
\par
We now show that $\widetilde H_{t,a}$ has at most $(\ell-1)(\ell-1)/2$ affine singular points $(u,v,1)$ satisfying $u\ne v$. Let $a\in\F$ be nonzero. Since $\ell\ge 3$ is odd, the first two conditions of~\eqref{eqn:cond_affine_sing_point} are not trivially satisfied. Thus we find from a degree argument that there are exactly $(\ell-1)(\ell-2)$ pairs $(u,v)$ with $u\ne v$ that satisfy the first two conditions of~\eqref{eqn:cond_affine_sing_point}. Notice that, if $(u,v)$ is such a pair, then $(u+1,v)$ also satisfies the first two conditions of~\eqref{eqn:cond_affine_sing_point}. Now let $(u,v,1)$ be a singular point of $\widetilde H_{t,a}$, so that the pair $(u,v)$ satisfies~\eqref{eqn:cond_affine_sing_point}. We claim that $(u+1,v,1)$ is not a singular point of $\widetilde H_{t,a}$, for if $(u+1,v)$ satisfies all three conditions of~\eqref{eqn:cond_affine_sing_point}, then $(u+1)^\ell=u^\ell$, which implies $a=0$ and so contradicts our assumption that $a$ is nonzero. Hence there are at most $(\ell-1)(\ell-2)/2$ affine singular points on $H_{t,a}$. 
\end{proof}
\par
We now show that $H_{t,a}$ has an absolutely irreducible factor for some $a\ne 1$ in $\A_n$ and all sufficiently large $n$.
\begin{proposition}
\label{pro:irreducible_factor_g}
Let $t$ be a number of the form $2^k\ell+1$ for integers $k\ge 1$ and odd $\ell\ge 3$. Then $H_{t,a}$ has an absolutely irreducible factor over $\F_{2^n}$ for some $a\ne 1$ in $\A_n$ and all sufficiently large $n$.
\end{proposition}
\par
To prove the proposition, we shall need one further standard result (see Hernando and McGuire~\cite[Lemma~10]{HerMcG2011}, for example).
\begin{lemma}
\label{lem:splitting_of_irreducible_polys}
Let $F\in\F_q[X_1,\dots,X_m]$ be a polynomial of degree $d$, irreducible over~$\F_q$. Then there exists a natural number $s\mid d$ such that, over its splitting field, $F$ splits into $s$ absolutely irreducible polynomials, each of degree $d/s$.
\end{lemma}
\par
\begin{proof}[Proof of Proposition~\ref{pro:irreducible_factor_g}]
If $H_{t,a}=AB$ is a nontrivial factorisation of $H_{t,a}$ and $A$ and $B$ are not relatively prime, then by definition, $\sum_P I_P(A,B)=\infty$. However, by Lemma~\ref{lem:number_of_sing_points} and Corollary~\ref{cor:sing_points}, $H_{t,a}$ has a finite number of singular points $P$, each having a finite intersection number $I_P(A,B)$. Hence we can assume that $A$ and $B$ are relatively prime, which allows us to use the conclusion of Bezout's Theorem.
\par
Write
\[
H_{t,a}=Q_1Q_2\cdots Q_r,
\]
where $Q_i$ is irreducible over $\F_{2^n}$. Let $d_i$ be the degree of $Q_i$. By Lemma~\ref{lem:splitting_of_irreducible_polys} there exist natural numbers $s_i$ such that $Q_i$ splits into $s_i$ absolutely irreducible factors over $\F$, each of degree $d_i/s_i$. If $s_i=1$ for some $i\in\{1,\dots,r\}$, then $H_{t,a}$ has an absolutely irreducible factor over $\F_{2^n}$ and we are done. Thus assume, for a contradiction, that $s_i>1$ for each $i\in\{1,\dots,r\}$.
\par
\par
We arrange the factors of $Q_i$ into three polynomials, $C_i$, $D_i$, and $R_i$, such that $\deg C_i=\deg D_i$ and such that $R_i=1$ if $s_i$ is even and $\deg R_i=d_i/s_i$ if $s_i$ is odd. Write $C=C_1\cdots C_r$, $D=D_1\cdots D_r$, and $R=R_1\cdots R_r$. Let $\delta$ be the degree of $C$ (and of $D$) and let $\rho$ be the degree of $R$. Since $CDR$ is a factorisation of $H_{t,a}$, which has degree $t-2$, we find that
\begin{equation}
2\delta+\rho=t-2,   \label{eqn:deg_CDR}
\end{equation}
and, since $s_i>1$,
\begin{equation}
\rho\le \frac{t-2}{3},   \label{eqn:deg_R}
\end{equation}
which gives
\[
(\deg CR)(\deg D)=(\delta+\rho)\delta=\frac{(2\delta+\rho)^2-\rho^2}{4}\ge \frac{2}{9}(t-2)^2.
\]
Bezout's theorem then gives
\begin{equation}
\sum_PI_P(CR,D)\ge \frac{2}{9}(t-2)^2.   \label{eqn:Bezout_kge2}
\end{equation}
On the other hand, we find from Lemma~\ref{lem:number_of_sing_points} and Corollary~\ref{cor:sing_points} that, for some $a\ne 1$ in $\A_n$ and for all sufficiently large $n$,
\[
\sum_P I_P(CR,D)\le \ell 4^{k-1}+\frac{(\ell-1)(\ell-2)}{2}2^k.
\]
This contradicts~\eqref{eqn:Bezout_kge2} for $k\ge 2$ since $\ell>1$.
\par
We now consider the case $k=1$, so that $t=2\ell+1$. Choose $a\ne 1$ in $\A_n$ and take $n$ sufficiently large so that the assertions of Corollary~\ref{cor:sing_points} hold. Since $k=1$, we find from Corollary~\ref{cor:sing_points} that all affine singular points of $H_{t,a}$ are of the form $(u,v,1)$ with $u\ne v$ and the multiplicity of such a singular point equals~$2$. Hence an affine singular point of $H_{t,a}$ can only be a point of at most two of the factors of $H_{t,a}$. Given two factors $F$ and $G$ of $H_{t,a}$, let $N_{FG}$ be the number of affine singular points of $H_{t,a}$ that are on both $F$ and~$G$. Then, by Corollary~\ref{cor:sing_points},
\begin{equation}
N_{CD}+N_{CR}+N_{DR}\le \frac{(\ell-1)(\ell-2)}{2}.   \label{eqn:number_sing_points_m2}
\end{equation}
Bezout's Theorem gives
\begin{align*}
\sum_PI_P(CD,R)&=2\delta\rho\\
\sum_PI_P(CR,D)&=(\delta+\rho)\delta\\
\sum_PI_P(DR,C)&=(\delta+\rho)\delta.
\end{align*}
We estimate the left hand sides using Lemma~\ref{lem:number_of_sing_points} and Corollary~\ref{cor:sing_points} and obtain
\begin{align*}
2(N_{CR}+N_{DR})+\ell & \ge 2\delta\rho\\
2(N_{CD}+N_{DR})+\ell & \ge (\delta+\rho)\delta\\
2(N_{CD}+N_{CR})+\ell & \ge (\delta+\rho)\delta.
\end{align*}
Summing these equations gives
\begin{align*}
2\delta^2+4\delta\rho&\le 4(N_{CD}+N_{CR}+N_{DR})+3\ell\le 2(\ell-1)(\ell-2)+3\ell,
\end{align*}
using~\eqref{eqn:number_sing_points_m2}. Since $t=2\ell+1$, we have from~\eqref{eqn:deg_CDR}
\[
\ell=\frac{2\delta+\rho+1}{2}
\]
and therefore
\begin{equation}
2\delta(2\rho+1)\le \rho(\rho-1)+6.   \label{eqn:alpha_bound}
\end{equation}
From~\eqref{eqn:deg_CDR} and~\eqref{eqn:deg_R} we conclude that $\delta\ge\rho$, so that
\[
2\rho(2\rho+1)\le \rho(\rho-1)+6
\]
or equivalently $\rho(\rho+1)\le 2$, forcing $\rho\le 1$. But, if $\rho=0$, then $t$ is even by~\eqref{eqn:deg_CDR}, a contradiction. Hence $\rho=1$ and then from~\eqref{eqn:alpha_bound} we find that $\delta=1$, giving $t=5$ by~\eqref{eqn:deg_CDR}. But $t=5$ cannot be written as $2\ell+1$ for odd $\ell$, which completes the proof.
\end{proof}
\par
Now our main result, Theorem~\ref{thm:exceptional_monomials_odd}, follows from Propositions~\ref{pro:irreducibility_implies_rational_points}, \ref{pro:abs_ireducible_l1}, and~\ref{pro:irreducible_factor_g}.


\section{Final remarks}

Since the submission of this paper, various new results have been obtained by other authors. Most notably, M\"uller and Zieve~\cite{MulZie2013} give a characterisation of low-degree planar monomials, thereby proving Conjecture~\ref{con:all_exceptional_planar_monomials} and providing a different proof of Theorem~\ref{thm:exceptional_monomials_odd}. New examples of planar functions, in particular planar binomials, have been found by Hu, Li, Zhang, Feng, and Ge in~\cite{HuLiZhaFenGe2013}.


\providecommand{\bysame}{\leavevmode\hbox to3em{\hrulefill}\thinspace}
\providecommand{\MR}{\relax\ifhmode\unskip\space\fi MR }
\providecommand{\MRhref}[2]{%
  \href{http://www.ams.org/mathscinet-getitem?mr=#1}{#2}
}
\providecommand{\href}[2]{#2}


\end{document}